\documentclass{amsart}
\usepackage[utf8]{inputenc}
\usepackage{tabu}
\usepackage{graphicx}
\usepackage{amssymb}
\usepackage{mathtools}
\usepackage{enumerate}
\usepackage{amsthm}

\newtheorem{theorem}{Theorem}

\newtheorem{corollary}{Corollary}
\newtheorem{proposition}{Proposition}
\newtheorem{lemma}{Lemma}

\title{Counting solutions to invariant equations in dense sets}
\author{Tomasz Kościuszko}

\usepackage[numbers]{natbib}
\setcitestyle{square}

\begin{document}
\maketitle
\begin{abstract}
We prove a lower bound of \(\exp(-C\log^7(2/\alpha))N^{k-1}\) to the number of solutions of an invariant equation in $k$ variables, contained in a set of density~\(\alpha\). Moreover, we give a Behrend-type construction for the same problem with the number of solutions of a convex equation bounded above by \(\exp(-c\log^2(2/\alpha))N^{k-1}\). Furthermore, improving the result of Schoen and Sisask, we show that if a set does not contain any non-trivial solutions to an equation of length \(k\geq 2\cdot 3^{m+1}+2\) for some positive integer $m$, then its size is at most \(\exp(-c\log^{1/(6+\gamma_m)} N)N\), where \(\gamma_m = 2^{1-m}\).
\end{abstract}
\section{Introduction}
Finding structure in dense sets of integers has been a challenge to mathematical research since Van der Waerden  proved his theorem on arithmetic progressions in 1927~\cite{vdw}. Of particular interests have been the quantitative results. We would like to have an upper-bound to the size of a set, which does not contain a certain structure. In the case of Roth's Theorem~\cite{roth}, we consider a set \(A\subseteq \{1,2,\cdots, N\}\) which contains no non-trivial solutions to the equation \[x+y=2z.\]
Roth proved that \[|A|\leq C \frac{N}{\log\log N}.\] After many improvements over the years, a sensational result of Kelley and Meka~\cite{kelley} showed a near-optimal bound \[|A|\leq \exp(-c\log^{1/11}N)N.\]

Other variations of this problem have been considered, for example, since Schoen and Shkredov~\cite{many} and the subsequent work of Schoen and Sisask~\cite{schoen} we know that longer equations like \[x_1+x_2+x_3=3y\] also restrict the size of the subset and even better bounds than the one from Kelley and Meka are known, namely \[|A|\leq \exp(-c\log^{1/7}N)N.\]

In this paper we offer an improvement to the result of Schoen and Sisask~\cite{schoen}, showing that long equations like \[x_1 + x_2 + \cdots + x_{k-1} = (k-1) x_k\] restrict the size of $A$ even more than 4-term equations. Below is a formulation of our result. By a trivial solution there, we mean one where all variables are equal. 
\begin{theorem}\label{thm:strong_no_solutions}
Let \(m\geq 1\) and \(k\geq 2\cdot 3^{m+1}+2\). Let \ \(A\subseteq \{1,2,\cdots, N\}\) be a set which contains no non-trivial solutions to the equation
\[x_1+x_2+\cdots + x_{k-1} = (k-1)x_k.\]
Then \(|A|\leq \exp(-c\log^{1/(6+\gamma_m)} N)N\), where \(\gamma_m = 2^{1-m}\).
\end{theorem}
For the sake of readability, we state Theorem \ref{thm:strong_no_solutions} for the specific equation of length~$k$. Our proof can be generalized to any invariant equation, in a similar way as in Theorem~\ref{thm:number_of_solutions}.  We say that a linear equation \[a_1 x_1 + a_2 x_2 + \cdots +a_k x_k =0\] with coefficients \(a_i \in \mathbb{Z}\) is invariant when \( a_1+a_2+\cdots a_k = 0\).

We also generalize the result of Schoen and Sisask~\cite{schoen} so that we can find an upper-bound not only if there are no non-trivial solutions, but also if their count is abnormally small. Bloom~\cite{bloom} already addressed this problem and he gave a proof of the following theorem.
\begin{theorem} (Bloom) \label{thm:number_of_solutions_bloom}
    Let \(A\subseteq \{1,2,\cdots, N\}\) be such that \(|A|=\alpha N\) and let \(a_1 x_1 + a_2 x_2 + \cdots +a_k x_k =0\) be an invariant equation in \(k\geq 3\) variables. Then for some large constant $C$, there are at least
    \[\exp(-C\alpha^{-1/(k-2)}\log^4(2/\alpha))N^{k-1}\]
    solutions to the equation.
\end{theorem}
Theorem~\ref{thm:number_of_solutions_bloom} is still the best published bound for 3-term equations, however Kelly and Meka~\cite{kelley} recently showed a much more efficient way of counting solutions to equations of length 3. They gave a lower bound of \[\exp(-C\log^{11}(2/\alpha))N^{2}.\] We significantly improve Bloom's bound for the 4-term and longer equations, our bound has a similar shape to the one from Kelley and Meka, more precisely we prove the following theorem.
\begin{theorem} \label{thm:number_of_solutions}
    Let \(A\subseteq \{1,2,\cdots, N\}\) be a set of size \(\alpha N\) and let \(a_1 x_1 + a_2 x_2 + \cdots +a_k x_k =0\) be an invariant equation in \(k\geq 4\) variables. Then for some large constant $C$, there are at least
    \[\exp(-C\log^7(2/\alpha))N^{k-1}\]
    solutions to the equation, that is tuples \((x_1, x_2,\cdots, x_k)\in A^k\), for which\\ \(a_1 x_1 + a_2 x_2 + \cdots +a_k x_k =0\).
\end{theorem}

The new bound can be used to boost results taking advantage of the Fourier Transference Principle~\cite{transference}. For example, it could be applied to the result by Prendiville on solving equations in dense Sidon sets~\cite{prend}. We state an improved result on Sidon sets in the ``Applications" section.

Finally, we give a construction which complements our Theorem \ref{thm:number_of_solutions}. We show a lower bound, similar to a well-known Behrend's construction~\cite{behrend}, where a set has high density, but contains only a few solutions to an invariant equation.

\begin{theorem}\label{thm:behrend_construction}
    Let \(k\geq 4\) and \(\alpha \ll 1\). There exist infinitely many integers \(N\geq 1\) and sets \(A\subseteq \{1,2,\cdots, N\}\) with \(|A|\geq \alpha N\), such that $A$ contains no more than
    \[\exp(-c\log^2(2/\alpha))N^{k-1}\] solutions to the equation \(x_1+\cdots+x_{k-1}=(k-1)x_k\).
\end{theorem}

We will prove Theorems~\ref{thm:strong_no_solutions} and \ref{thm:number_of_solutions} by treating $A$ as a subset of the group \(\mathbb{Z}/p\mathbb{Z}\) instead of the interval \(\{1,2,\cdots, N\}\). If \(p > (|a_1|+|a_2|+\cdots+|a_n|) N\), then no solutions in the integer case imply no solutions in the \(\mathbb{Z}/p\mathbb{Z}\) case and therefore one version implies the other. 

\section{Notation} 

Let us fix some notation and recall a couple of well-known definitions.
By $c$ and $C$ we mean real, positive constants, where $c$ is sufficiently small and $C$ is sufficiently large for all of our uses. If we wanted to be really precise, we would have to call the constants \(c_1,c_2,\cdots\) and \(C_1, C_2, \cdots\) in various parts of different proofs, however for simplicity we omit the indices.
We work in the group \(\mathbb{Z}/p\mathbb{Z}\), however most of the definitions are given for a general, finite, abelian group $G$ of size $N$ and then applied to \(\mathbb{Z}/p\mathbb{Z}\).
Let $A$ be a subset of $G$. Define a normalized indicator function to be
\[\mu_A = 1_A \cdot \frac{1}{|A|},\]
where \(1_A(x)\) is the function that gives $1$ when \(x\in A\) and $0$ otherwise.\\
By the convolution of two functions \(f,g:G\rightarrow \mathbb{R}\) we mean
\[f*g(x) = \sum_{t\in G} f(t)g(x-t).\]
We sometimes write \(f^{(k)}(x)\) to mean multiple convolutions, that is 
\[f^{(k)}(x) = f*f*\cdots*f(x) \text{ where $f$ appears $k$ times}.\]
Let \(1\leq p < \infty\), the \(L_p\) norm of a function \(f:G\rightarrow \mathbb{R}\) is defined as
\[||f||_p = \Bigl( \sum_{x\in G} |f(x)|^p\Bigr)^{1/p},\]
when \(p=\infty\) we always mean \(||f||_{\infty} = \max_{x\in G}|f(x)|\). 
For a function \(f:G\rightarrow \mathbb{R}\) we define expectation as
\[\mathbb{E}_{x\in G} f(x) = \frac{1}{|G|}\sum_{x\in G}f(x).\]
Denote the group of all characters (homomorphisms) \(\gamma: G\rightarrow \mathbb{C}\) as \(\widehat{G}\). Given a function \(f:G\rightarrow \mathbb{R}\) we define a Fourier coefficient at \(\gamma \in \widehat{G}\) as
\[\widehat{f}(\gamma) = \sum_{x\in G}f(x)\overline{\gamma(x)}.\]
We call the function \(\widehat{f}\) the Fourier Transform of \(f\).
\section{Tools for finding Almost-Periods}
A common technique in Additive Combinatorics is solving a problem in the group \(\mathbb{F}_p^n\), before stating it for a general group or an interval of integers. The advantage of \(\mathbb{F}_p^n\) is that we can make use of subspaces, which are not found in any group. Fortunately, Bohr sets act as approximate subspaces in any finite group $G$. Translating the ideas from the language of subspaces to the language of Bohr sets is usually possible, although quite technical. In our work we immediately present the reasoning by using Bohr sets. The paper by Schoen and Sisask~\cite{schoen} contains simpler proofs of their result in the case of \(\mathbb{F}_p^n\) as well as the general proofs. A reader unfamiliar with Bohr sets, could consider reading that paper as an introduction.

We record a couple of auxiliary definitions and propositions concerning the properties of Bohr sets. For more background on Bohr sets we recommend to the reader a book by Tao and Vu~\cite{taovu}.\\
Let \(0 < \rho \leq 2\) and let \(\Gamma \subseteq \widehat{G}\) for some finite group $G$. The Bohr set \(B=\mathrm{Bohr}(\Gamma, \rho)\) is defined as
\[B = \{x\in G : |1-\gamma(x)|\leq \rho \text{ for all }\gamma\in\Gamma\}.\]
We refer to \(\rho\) as the width and to \(|\Gamma|\) as the dimension of $B$.\\
If \(\delta > 0\) and \(B=\mathrm{Bohr}(\Gamma, \rho)\) we write \(B_{\delta}\) for \(B=\mathrm{Bohr}(\Gamma, \rho\delta)\) and call it a dilate of $B$.
A Bohr set $B$ with dimension $d$ is called regular when
\[1-12d|\delta|\leq \frac{|B_{1+\delta}|}{|B|}\leq 1+12d|\delta|,\]
for every \(|\delta|\leq 1/12d\).
\begin{proposition}\label{prop:bohr_set_convolution}
Let $B$ be a regular Bohr set of dimension $d$ and let \(B'\subseteq B_{\delta}\) where \(\delta \leq \epsilon/24d\), then we have
\[||\mu_B*\mu_{B'} - \mu_{B}||_1 \leq \epsilon.\]
\end{proposition}
\begin{proof}
Using the triangle inequality we notice that
\[||\mu_B*\mu_{B'} - \mu_B||_1 = \frac{1}{|B|}||1_B*\mu_{B'} - 1_B||_1\leq \frac{1}{|B|}\Bigl(||1_B*\mu_{B'}-1_{B+B'}||_1 + ||1_{B+B'}- 1_B||_1\Bigr).\]
Now, by regularity of $B$ we obtain
\[||1_{B+B'}-1_B||_1 = |(B+B')\setminus B| \leq \epsilon|B|/2.\]
Again by regularity of $B$ we have
\[||1_B*\mu_{B'}-1_{B+B'}||_1 = \sum_{x\in B+B'} 1 -\frac{1}{|B|}1_B*1_{B'}(x) = |B+B'| - |B| =\]\[= |(B+B')\setminus B|  \leq \epsilon|B|/2,\]
which ends the proof.
\end{proof}
We will make also use of the following two properties, for the proofs see the book by Tao and Vu~\cite{taovu}.
\begin{proposition}\label{regular_sub_bohr_set}
Let $B$ be a Bohr set. There exists \(\delta \in [\frac{1}{2}, 1]\), such that \(B_{\delta}\) is regular.   
\end{proposition}
\begin{proposition}\label{bohr_set_size_lower_bound}
Let $B$ be a Bohr set of dimension $d$ and let \(\delta > 0\). Then we have
\[|B_{\delta}|\geq (\delta/2)^{3d}|B|.\]
\end{proposition}
In the proof of Theorem~\ref{thm:number_of_solutions} we will follow a classical paradigm of finding density increments. The main tool is Croot-Sisask lemma, more precisely its version for Bohr sets proved by the means of Chang-Sanders lemma. Similarly as in~\cite{schoen}, we combine this powerful result with the fact, that for Bohr sets, almost-periods of convolutions and density increments are very closely related. That is depicted by the following lemma.
\begin{lemma} \label{lemma:almost_periods_to_increment}
Let \(\epsilon > 0\), \(f:G\rightarrow \mathbb{R}\) and let \(A\subseteq G\) have the size \(\alpha N\). Suppose that $B$ is a Bohr set, such that for every \(t\in B\)
\[||f*1_A(\cdot + t) - f*1_A||_{\infty} \leq \epsilon \text{ holds.}\]
Further assume that \(||f||_1\leq 1/(2\alpha)\) and \(f*1_A(0)\geq 1-\epsilon\). Then there exists a translate of $A$ (say $x+A$), such that \(B\cap (x+A)\) has density at least \(2\alpha(1-2\epsilon)\) inside $B$.
\end{lemma}
\begin{proof}
    We notice that since $B$ is symmetric around 0, we have
    \begin{equation*}
    \begin{split}
    ||f*1_A*\mu_B - f*1_A||_{\infty} &= \bigg|\bigg|\sum_{t\in -B}\Bigl(f*1_A(\cdot-t)\Bigr) \cdot\mu_B(t) - f*1_A\bigg|\bigg|_{\infty} \\
    &\leq \sum_{t\in B} ||f*1_A(\cdot+t) - f*1_A||_{\infty} \cdot \mu_B(t)\\
    &\leq \epsilon
    \end{split}
    \end{equation*}
    By the triangle inequality it follows that \(f*1_A*\mu_B(0) \geq 1-2\epsilon\). We now notice that
    \[||1_A*\mu_B||_{\infty}/(2\alpha) \geq ||f||_1||1_A*\mu_B||_{\infty} \geq f*1_A*\mu_B(0) \geq 1-2\epsilon.\]
    Therefore, for some $x$ we have \(1_A*\mu(x) \geq 2\alpha(1-2\epsilon)\) and we have proved the result as \(1_A*\mu_B(x) = \frac{1}{|B|}|(x+A)\cap B|\).
\end{proof}
Let us now state the two main results we will be using in our proof. They are both stated in~\cite{schoen} (Theorem 2.1 and Proposition 5.3) and the second one is proved in~\cite{chang} (Proposition 4.2).
\begin{theorem} \label{thm:croot_sisask}
(Croot-Sisask) Fix constants \(\epsilon \in (0,1)\), \(k\in \mathbb{N}\) and \( p\geq 2 \). Let \(A, L, S\) be subsets of a finite abelian group and suppose that \(|A+S|\leq K|A|\).
There exists a set \(T\subseteq S\) of size at least  \(|T| \geq 0.99K^{-C p k^2 / \epsilon^2}|S|\), such that for every \(t\in kT-kT\) we have
\[||1_A*1_L(\cdot + t) - 1_A*1_L||_{p}\leq \epsilon |A| |L|^{1/p}.\]
\end{theorem}
The set $kT-kT$ is referred to as Almost Periods of \(1_A*1_L\), because such function does not change by much when shifted by any one element of \(kT-kT\). Our aim is to prove an analogue result, to be used in Lemma~\ref{lemma:almost_periods_to_increment}. For that, we shall construct a Bohr set out of the initial set of almost periods (see later in Theorem~\ref{thm:croot_sisask_modified}).
The next result is Chang's Theorem. We use the version proposed by Sanders. It concerns the large spectrum of \(1_X\) defined as
\[\mathrm{Spec}_{\delta}(1_X) = \{\gamma \in \widehat{G} : |\widehat{1}_X(\gamma)| \geq \delta |X|\}\]
and asserts that it has low dimension, in the sense that we can find a low-dimensional Bohr set on which the values of the characters from \(\mathrm{Spec}_{\delta}(1_X)\) are very close to 1.
\begin{theorem}\label{thm:chang_sanders}
    (Chang-Sanders)
    Fix \(\nu>0\). Suppose that \(B = \mathrm{Bohr}(\Gamma, \rho)\) is regular and set \(d = |\Gamma|\). Let \(X\subseteq B\) be a subset of density \(\sigma\). Then we can find a set of characters \(\Lambda\) and radius \(\rho_2\) such that 
    \[|\Lambda|\leq C \log(2/\sigma)\text{, }\rho_2\geq c \rho \nu \delta^2 /d^2 \log(2/\sigma)\]
    and for every \(\gamma \in \mathrm{Spec}_{\delta}(1_X)\) we have
    \[|1-\gamma(t)|\leq \nu \text{ for all } t\in \mathrm{Bohr}(\Lambda \cup \Gamma, \rho_2).\]
\end{theorem}
To prove Theorem~\ref{thm:croot_sisask_modified} that will be used to find almost periods as needed in Lemma~\ref{lemma:almost_periods_to_increment} we first show a corollary to the Croot-Sisask Lemma that will enable us to consider multiple sets instead of just two. In the proof we take advantage of Young's inequality, which we state here.
\begin{theorem}\label{thm:young}(Young's convolution inequality - special case)
Let \(f,g:G\rightarrow \mathbb{R}\) be functions and let \(q\geq 1\), then we have
\[||f*g||_q\leq ||f||_q ||g||_1.\]
\end{theorem}
\begin{corollary}\label{cor:croot_sisask}
Fix constants \(\epsilon \in (0,1)\) and \( k\in \mathbb{N}\). Let \(A_1, A_2, \cdots, A_n, M, L, S\) be subsets of a finite abelian group. Suppose that \(|A_1 + S| \leq K|A_1|\) and that \(\eta = |M|/|L| \leq 1\). Then we can find a set \(T\subseteq S\) such that
\[|T| \geq \exp(-Ck^2\epsilon^{-2}\log(2/\eta)\log(2K))|S|\] and for every \(t\in kT-kT\)
\[||1_{A_1}*\cdots *1_{A_n}*1_M*1_L(\cdot + t) - 1_{A_1}*\cdots *1_{A_n}*1_M*1_L||_{\infty} \leq \epsilon |A_1|\cdots|A_n||M|.\]
\end{corollary}
\begin{proof}
Let \(f = 1_{A_2}*1_{A_3}*\cdots *1_{A_n}\). By writing the definition of convolution and using Holder's inequality we see that for any \(t\in G\) and for any \(p,q\geq 1\) such that \(\frac{1}{p}+\frac{1}{q}=1\) we have
\[||1_{A_1}*1_{M}*f*1_{L}(\cdot + t) - 1_{A_1}*1_{M}*f*1_{L}||_{\infty} \leq ||1_{A_1}*1_{M}(\cdot + t) - 1_{A_1}*1_M||_q ||f*1_L||_p.\]
Let us set \(p=\log(2/\eta)\). Using Theorem~\ref{thm:croot_sisask} with \(\epsilon/3\) we get a set $T$ of desired size such that for any 
\(t\in T\) there is
\begin{equation*}
\begin{split}
    ||1_{A_1}*1_{M}(\cdot + t) - 1_{A_1}*1_M||_q ||f*1_L||_p &\leq \frac{1}{3}\epsilon |A_1| |M|^{1/q} ||f*1_L||_p\\
    &\leq \frac{1}{3}\epsilon|A_1||A_2|\cdots |A_n| |M|^{1/q} |L|^{1/p}\\
    &=\frac{1}{3}\epsilon|A_1||A_2|\cdots |A_n| |M|(|L|/|M|)^{1/p},
\end{split}
\end{equation*}
where the second inequality follows by applying Theorem~\ref{thm:young}. The Corollary is proved because for our choice of $p$ we have
\[\frac{1}{3}(|L|/|M|)^{1/p} = \frac{1}{3}(1/\eta)^{1/\log(2/\eta)} \leq 1.\]
\end{proof}
We can now prove the version of the Croot-Sisask lemma that will allow us to find almost periods which form a large Bohr set. This is a version of Theorem 5.4 from the paper of Schoen and Sisask~\cite{schoen}.
\begin{theorem}\label{thm:croot_sisask_modified}
    Fix \(\epsilon \in (0,1)\). Let \(A_1, A_2, \cdots, A_n, M, L\) be subsets of $G$. Let $B$ be a regular Bohr set of dimension $d$ and width \(\rho\). Suppose that there exists \(S\subseteq B\), such that \(|A_1+S|\leq K|A_1|\). Denote the density of $S$ in $B$ as \(\sigma\). Moreover, suppose that \(\sigma > 0\) and \(\eta = |M|/|L| \leq 1\). Then there exists a Bohr set \(B'\subset B\) with the property that for every \(t\in B'\) we have
    \[||1_{A_1}*\cdots *1_{A_n}*1_L*1_M(\cdot + t) - 1_{A_1}*\cdots *1_{A_n}*1_L*1_M||_{\infty} \leq \epsilon |A_1|\cdots|A_n||M|.\]
    Furthermore, $B'$ can be taken to have width at least \(\rho\epsilon\eta^{1/2}/(d^2d')\) and dimension at most $d+d'$ where
    \[d'\ll \epsilon^{-2}\log^2(2/\epsilon\eta)\log(2/\eta)\log(2K)+\log(1/\sigma).\]
    
\end{theorem}
    \begin{proof}
        To simplify the notation, let us write
        \[g = 1_{A_1}*\cdots *1_{A_n}*1_L*1_M(\cdot + t)\]
        We start by applying Corollary~\ref{cor:croot_sisask} to obtain a set of almost periods $T$, with \(|T|\geq \exp(-Ck^2\epsilon^{-2}\log(2/\eta)\log(2K))|S|\), such that for every \(t\in kT-kT\) we have
        \[||g(\cdot + t) - g||_{\infty} \leq \epsilon_2 |A_1|\cdots|A_n||M|,\tag{1}\]
        Where \(\epsilon_2=\epsilon/4\).
        We can notice that although not a Bohr set, our set of almost periods \(kT-kT\) is already highly structured, especially if $k$ is large. We will make use of this observation to find a large Bohr set \(B'\) which is ``close" to \(kT-kT\).
        We will approach this problem from the perspective of Fourier Analysis, which will allow us to use Theorem~\ref{thm:chang_sanders}. For that let us fix any \(z\in T\) and set \(X=T-z\). Then clearly the inequality holds for any \(t\in k X\) as \(kX = kT-kz \subseteq kT - kT\). We notice that the following functions can be approximated by one another.
        \begin{equation*}
        \begin{split}
        ||g-g*\mu_X^{(k)}||_{\infty} &= \bigg|\bigg|\sum_{(x_1,\cdots,x_k)\in X^k}(g(\cdot)-g(\cdot -x_1-\cdots -x_k))\mu_X(x_1)\cdots\mu_X(x_k)\bigg|\bigg|_{\infty}\\
        &\leq \sum_{(x_1,\cdots,x_k)\in X^k} ||g(\cdot)-g(\cdot -x_1-\cdots -x_k)||_{\infty}\mu_X(x_1)\cdots\mu_X(x_k)\\
        &\leq \epsilon_2 |A_1|\cdots|A_n||M|,           
        \end{split}
        \end{equation*}
        where the last inequality follows by (1) upon noticing that \(x_1+x_2+\cdots+ x_k\in kT-kT\). Therefore if we manage to find a Bohr set \(B'\), such that for every \(t\in B'\) there is
        \[||g*\mu_X^{(k)}(\cdot + t) - g*\mu_X^{(k)}||_{\infty} \leq \epsilon_2 |A_1|\cdots|A_n||M| \tag{2}\]
        we can use the triangle inequality to finish the proof.\\
        Noticing that the conclusion of Theorem~\ref{thm:chang_sanders} remains true if we consider translates of Bohr sets, we apply it to $X$ as a subset of $B-z$. We show that the Bohr set produced is sufficient as $B'$. First of all let us check that its dimension and radius is good enough. We know that the density of $X$ in $B-z$ is
        \[\sigma_x = \exp(-Ck^2\epsilon_2^{-2}\log(2/\eta)\log(2K))\sigma .\]
        Substituting it in Theorem~\ref{thm:chang_sanders} we obtain a set of characters \(\Lambda\) with
        \[d'=|\Lambda|\leq Ck^2\epsilon_2^{-2}\log(2/\eta)\log{2K}+\log(1/\sigma)\]
        and \(\rho_2\) such that
        \[\rho_2\geq \frac{c\rho\nu\delta^2}{Cd^2 k^2\epsilon_2^{-2}\log(2/\eta)\log{2K}+\log(1/\sigma)}.\]
        We set \(k = C \lceil \log(2/\epsilon_2\eta)\rceil\), \(\delta=1/2\), \(\nu=\epsilon_2\eta^{1/2}\). With such setup the claimed constraints on $B'$ are satisfied.
        Let us verify that \(\mathrm{Bohr}(\Lambda \cup \Gamma, \rho_2)\) indeed contains almost periods of the relevant function. Fix its element $t$ and an element of the group $x$.
        Let us consider aforementioned inequality and write it by the Fourier inversion formula as
        \begin{equation*}
        \begin{split}
        \bigl|g*\mu_X^{(k)}(x + t) - g*\mu_X^{(k)}\bigr| &= \bigl|\mathbb{E}_{\gamma \in \widehat{G}}\widehat{g}(\gamma)\widehat{\mu}_X(\gamma)^k(\gamma (x+t) - \gamma(x))\bigr|\\
        &\leq \mathbb{E}_{\gamma \in \widehat{G}}|\widehat{g}(\gamma)||\widehat{\mu}_X(\gamma)|^k |1 - \gamma(t)|. 
        \end{split}
        \end{equation*}
        At this point let us consider two cases: either \(\gamma\in \mathrm{Spec}_\delta(1_X)\) or not. For the second case, we have by the definition of \(\mathrm{Spec}_\delta(1_X)\) that \[|\widehat{\mu}_X(\gamma)|^k = \widehat{1}_X(\gamma)^k/|X|^k \leq \delta ^ k \leq \Bigl(\frac{1}{2}\Bigr)^k\leq \epsilon_2\eta^{1/2},\] provided the constant $C$ is sufficiently large. Otherwise, for every \(t\in B'\) we have by Theorem~\ref{thm:chang_sanders} that \(|1-\gamma(t)|\leq \nu = \epsilon_2\eta^{1/2}\).
        Either way we have
        \[|\widehat{\mu}_X(\gamma)|^k|1-\gamma(t)|\leq \epsilon_2\eta^{1/2}.\]
        Thus the expectation can be estimated as
        \begin{equation*}
        \begin{split}
        \mathbb{E}_{\gamma \in \widehat{G}}|\widehat{g}(\gamma)||\widehat{\mu_X}(\gamma)|^k |1 - \gamma(t)| &\leq \epsilon_2\eta^{1/2} \cdot \mathbb{E}_{\gamma \in \widehat{G}}|\widehat{g}(\gamma)|\\
        &= \epsilon_2\eta^{1/2} \cdot \mathbb{E}_{\gamma \in \widehat{G}}|\widehat{1}_{A_1}(\gamma)|\cdots |\widehat{1}_{A_n}(\gamma)| |\widehat{1}_{L}(\gamma)| |\widehat{1}_{M}(\gamma)|\\
        &\leq \epsilon_2\eta^{1/2} |A_1|\cdots |A_n|\cdot\mathbb{E}_{\gamma \in \widehat{G}}|\widehat{1}_{M}(\gamma)| |\widehat{1}_{L}(\gamma)| \\
        &\leq  \epsilon_2\eta^{1/2} |A_1|\cdots |A_n||M|^{1/2}|L|^{1/2} \\
        &\leq\epsilon_2 |A_1|\cdots |A_n||M|,
        \end{split}
        \end{equation*}
        as required. We proved (2) for every \(t\in B'\), from which the theorem follows by the triangle inequality as shown before.
    \end{proof}
    \section{Obtaining a density increment}
    So far we have shown in Lemma~\ref{lemma:almost_periods_to_increment}, that a large Bohr set of almost periods lets us find a density increment inside of this Bohr set (upon translating the original set). We also know how to find such Bohr set using Theorem~\ref{thm:croot_sisask_modified}. It remains to show how to proceed so that the assumptions of Theorem~\ref{thm:croot_sisask_modified} are satisfied. This will require some non-trivial manipulations on Bohr sets. Let \(A\subseteq B\) be a set which contains few solutions to our fixed invariant equation and let $B$ be a regular Bohr set. We would like to dilate $B$ by some factor \(\delta\) and still be able to find a tanslate of $A$, which has high density within \(B_{\delta}\). The following lemma tells us that such translate can be found.
    
    \begin{lemma}\label{lemma:high_density_small_bohr_set}
        Let \(A\subseteq B\) with \(|A|=\alpha |B|\), where $B$ is a regular Bohr set of dimension $d$ and radius \(\rho\). Suppose that \(\delta\leq \frac{\alpha}{240d}\) is a constant, such that \(|B_{1+\delta}|\leq 1.01 |B|\).
        There exists \(x\in G\) such that \(|A\cap (x+B_{\delta})|\geq 0.9\alpha |B_{\delta}|\).
    \end{lemma}
    \begin{proof}
    Applying Proposition 1 to \(B_{\delta}\) we have
    \[||\mu_B*\mu_{B_{\delta}}(x) - \mu_B(x)||_1\leq \alpha/10,\]
    and by the triangle inequality we get
    \begin{equation*}
    \begin{split}
    \alpha = \mu_B(A) &\leq ||\mu_B*1_A(x) - \mu_B*\mu_{B_{\delta}}*1_A(x)||_1 + ||\mu_B*\mu_{B_{\delta}}*1_A(x)||_1\\
    &\leq \alpha/10 + \frac{1}{|B|}\sum_{x\in B}\mu_{B_{\delta}}*1_A(x).
    \end{split}
    \end{equation*}
    Thus for some \(x\in B\) we have
    \[\mu_{B_{\delta}}*1_A(x) \geq 0.9\alpha\]
    as required.
    \end{proof}
    We now generalize Lemma~\ref{lemma:high_density_small_bohr_set} to allow multiple factors \(\delta_i\) and multiple coefficients $a_i$, for which we consider \(a_i\cdot A\). In this lemma we have to assume that our group is \(\mathbb{Z}/p\mathbb{Z}\). The reason is that we want to define the operation of multiplying a Bohr set by an element of the group. Let \(B=\mathrm{Bohr}(\Gamma, \rho))\) and \(a\in \mathbb{Z}/p\mathbb{Z}\), we define
    \[aB:= \{x\in \mathbb{Z}/p\mathbb{Z} : |1-\gamma(a^{-1}x)|\leq \rho \text{ for all } \gamma\in \Gamma\}.\]
    So if \(\gamma\) is a generator of $B$, then \(\gamma^{a^{-1}}\) is a generator of \(aB\). This way if \(x\in B\) then \(ax\in aB\) and \(aB\) is a Bohr set of the same dimension and radius as~$B$.
    \begin{lemma} \label{lemma:coefficients}
        Let \(B\subseteq \mathbb{Z}/p\mathbb{Z}\) be a regular Bohr set of dimension $d$ and radius \(\rho\). Let \(A\subseteq B\) be its subset of size \(\alpha |B|\). Let \(a_1, a_2, \cdots, a_k\) be non-zero integers and \(\delta_1, \delta_2,\cdots \delta_k\) numbers from the interval $(0,1]$. There exist sets \(A_1,A_2,\cdots A_k \subseteq A-x\) for some translate $x$ and a Bohr set \(B'\) such that \(a_i \cdot A_i \subseteq B'_{\delta_i}\)  and either \[|(a_i\cdot A_i) \cap B'_{\delta_i}| \geq \frac{7}{8}\alpha |B'_{\delta_i}| \text{ for all $i$}\] or \[|(a_i\cdot A_i) \cap B'_{\delta_i}| \geq (1 + 1/16k)\alpha |B'_{\delta_i}| \text{ for some $i$}.\] Moreover, $B'$ can be chosen so that its dimension is $d$ and its radius is \(\rho_2 \geq \rho\frac{c \alpha}{k d}.\)
    \end{lemma}
    \begin{proof}
        Let \(\epsilon := \frac{1}{16k}\alpha (\Pi_j |a_j|)^{-1} / 24d\),  \(B^i := \Bigl(\Pi_{j\neq i} a_j\Bigr) B_{\epsilon\cdot \delta_i}\) and \(B' := \Bigl(\Pi_{j} a_j \Bigr) B_{\epsilon}\). Clearly $B'$ satisfies the conditions on the dimension and the radius. \\
        Notice, that from Proposition 1 we have
        \[||\mu_B*\mu_{B^{i}} - \mu_B||_1 \leq  \frac{1}{16k}\alpha.\]
        Since \(\mu_B(A) = \alpha\) and by the application of the triangle inequality we get
        \begin{equation*}
        \begin{split}
        k\alpha\leq \sum_{i=1}^k \mu_B(A) &\leq \sum_{i=1}^k ||\mu_B*1_A(x) - \mu_B*\mu_{B^i}*1_A(x)||_1 +
        \sum_{i=1}^k ||\mu_B*\mu_{B^i}*1_A(x)||_1\\
        &\leq \frac{1}{16}\alpha + \frac{1}{|B|}\sum_{i=1}^k\sum_{x\in B} \mu_{B^i}*1_A(x).
        \end{split}
        \end{equation*}
        Thus, for some \(x\in B\) the sum is at least equal to the average, so
        \[\sum_{i=1}^k\mu_{B^i}*1_A(x) \geq (k-1/16)\alpha.\]
        By our assumption we have \(||\mu_{B^i}*1_A(x)||_{\infty}\leq (1 + 1/16k)\alpha\), thus
        \begin{equation*}
        \begin{split}
        \mu_{B^i}*1_A(x) &\geq (k-1/16)\alpha - \sum_{j\neq i} \mu_{B^i}*1_A(x)\\
        &\geq (k-1/16)\alpha - (k-1)(1 + 1/16k)\alpha\\
        &\geq \Bigl( 1 - \frac{2}{16} + \frac{1}{16k}\Bigr)\alpha\\
        &\geq \frac{7}{8}\alpha.
        \end{split}
        \end{equation*}
        Set \(A_i = (A-x) \cap B^i\). We clearly have \(a_i\cdot A_i \subseteq a_i  B^i\), therefore \(a_i\cdot A_i\subseteq B'_{\delta_i}\) for all $i$.
    \end{proof}
    All results which use Lemma~\ref{lemma:coefficients} will be also stated for the group \(\mathbb{Z}/p\mathbb{Z}\).
    In the next lemma we prove a density increment on a Bohr set with dimension and width slightly smaller than the initial one. A new idea here is applying Theorem~\ref{thm:croot_sisask_modified} to the set of Popular sums, as suggested by Sanders and Prendiville~\cite{prend}.
    \begin{lemma}\label{lemma:increment_on_bohr_set}
        Let \(B\subseteq \mathbb{Z}/p\mathbb{Z}\) be a regular Bohr set of dimension $d$ and radius \(\rho\) and let \(A\subseteq B\) be its subset of size \(\alpha |B|\). Let \[a_1 x_1 + a_2 x_2 + \cdots +a_k x_k =0\] be an invariant equation in \(k\geq 4\) variables and suppose that the number of solutions to in $A$ does not exceed \[\exp(-Cd (\log(d/\alpha)))|A|^{k-1}.\] Then we can find a Bohr set \(B^*\) of dimension \(d+d'\) and radius \(\rho_2\), such that for some $x$ we have \(B^*\cap (A-x) \geq (1+1/16k)\alpha |B^*|.\) Moreover, $B^*$ can be chosen so that
        \[d' \leq C\log^4(2/\alpha)\]
        and 
        \[\rho_2\geq \rho\alpha^{3/2}/(d^5d')\]
    \end{lemma}
    \begin{proof}
        We will set up the proof by using Lemma~\ref{lemma:coefficients} with \(a_i\) - the coefficients of the equation. Let $B'$ be the Bohr set defined at the beginning of the proof of the Lemma, explicitly \(B' := \Bigl(\Pi_{j} a_j \Bigr) B_{\epsilon}\) for \(\epsilon := \frac{1}{16k}\alpha (\Pi_j |a_j|)^{-1} / 24d\). By Proposition 2 it is possible to choose \(\delta \geq \frac{1}{Cd}\) so that \(|B'_{1+(k-3)\delta}| \leq 1.01 |B'|\) and \(B'_{\delta}\) is regular. We are ready to use Lemma~\ref{lemma:coefficients}, specifying \(r_1 = r_3 = 1\), \(r_2 = \delta /m \) and \( r_4 = r_5 = \cdots = r_k = \delta\), where we set \(m:=\log (2/\alpha)\). If the second conclusion is true, that is for some $i$ there is \(|(a_i\cdot A_i) \cap B'_{r_i}| \geq (1 + 1/16k)\alpha |B'_{r_i}|\) we immediately finish the proof, as we have the desired density increment with \(B^* = B'_{r_i}\).
        Otherwise we continue the proof, keeping the first conclusion.
        We now define a set of Popular sums $P$. Consider the function
        \[f(x) = 1_{a_3\cdot A_3}*1_{a_4\cdot A_4}*\cdots *1_{a_k\cdot A_k}(x).\]
        We see that
\[\mathrm{supp} f = a_3\cdot A_3 + a_4\cdot A_4 + \cdots + a_k\cdot A_k \subseteq B' + (k-3)B'_{\delta}\subseteq B'_{1+(k-3)\delta}.\]
We fix a threshold to be
\[Q := \frac{\alpha}{8}|A_4||A_5|\cdots|A_k|\]
and finally define \(P\subseteq B'_{1+(k-3)\delta}\) as
\[P:=\{x\in B'_{1+(k-3)\delta} : f(x)\geq Q\}.\]
We will consider two cases, depending on the size of $P$. Let us first assume that \(|P|\geq | B'_{1+(k-3)\delta}|/2\) and look at the other possibility later. We will apply Theorem~\ref{thm:croot_sisask_modified} to the sets, \(M:=a_1\cdot A_1\), \(A :=a_2\cdot A_2\) and \(L=P^c = B'_{1+(k-3)\delta} \setminus P\). We define $S$ to be \(a_2\cdot B'_{\delta\nu}\) where \(\nu:=1/Cd\), to get a regular Bohr set with radius at least \(\rho/Cd^3\), such that \(|B'_{\delta(1+\nu)}|\leq 2|B'_{\delta}|\).
We have to verify the assumptions of Theorem~\ref{thm:croot_sisask_modified}, let us start by calculating 
\[|a_2\cdot A_2 + S| \leq |B'_{\delta(1+\nu)}|\leq 2|B'_{\delta}| = 2 |a_2 \cdot B^{(2)}|\leq 2 \cdot \frac{8}{7\alpha}|A_2| =\frac{16}{7\alpha}|a_2\cdot A_2|,\]
and so \(K = \frac{16}{7\alpha}\) is sufficient. To calculate \(\eta\), let us see that
\[|L| \leq (1 - 1/2)|B'_{1 + (k-3)\delta}|\leq 1.01 \cdot |B'| /2 = 0.505 |a_1 \cdot B^{(1)}|\leq \frac{2}{3\alpha}|a_1\cdot A_1|,\]
so \(\eta = \alpha \leq 1\) is enough.\\
Therefore from Theorem~\ref{thm:croot_sisask_modified}, we get a Bohr set $B^*$, such that for every \(t\in B^*\) there is 
\[||1_{a_1\cdot A_1}*1_{a_2\cdot A_2}*1_L(\cdot+t) - 1_{a_1\cdot A_1}*1_{a_2\cdot A_2}*1_L||_{\infty}\leq \epsilon |A_1||A_2|.\]
We now show that Lemma~\ref{lemma:almost_periods_to_increment} can be used to obtain a density increment. We easily see that \[||1_{a_2\cdot A_2}*1_{L}||_1 = |A_2||L|\leq |A_1||A_2|
/(2\alpha)\]
and so we take our function to be \(1_{a_2\cdot A_2}*1_{L}/|A_1||A_2|\).
To show that the remaining assumption is satisfied we need
\[1_{a_1\cdot A_1}*1_{a_2\cdot A_2}*1_{L}(0)\geq (1-\epsilon)|A_1||A_2|.\] Let us use what we know about the number of solutions in $A$. We notice that
\[1_{a_1\cdot A_1}* 1_{a_2\cdot A_2} * 1_P(0) \cdot Q \leq \exp(-Cd\log(d/\alpha))|A|^{k-1}.\tag{3}\]
By Proposition 3 we also see that
\[|A_i|\geq \frac{7\alpha}{8}|B'_{\delta}|\geq \Bigl(c\frac{\alpha}{d^2k}\Bigr)^{3d} |B|.\]
Simplifying the constants we get
\[|A_i|\geq \exp(- C' 3d \log(d^2k/\alpha))|A| \geq \exp(-Cd\log(d/\alpha))|A|.\] Applying this inequality multiple times to (3) we have
\begin{equation*}
\begin{split}
1_{a_1\cdot A_1}* 1_{a_2\cdot A_2} * 1_P(0) \cdot \frac{\alpha}{8} |A|^{k-3} &\leq 1_{a_1\cdot A_1}* 1_{a_2\cdot A_2} * 1_P(0) \cdot Q \cdot \exp(C d\log (d/\alpha))\\
&\leq \frac{\alpha }{64} |A|^{k-1},
\end{split}
\end{equation*}
where the last inequality follows from the restriction on the number of solutions, provided the constant $C$ has been chosen large enough. Thus we have
\[1_{a_1\cdot A_1}* 1_{a_2\cdot A_2} * 1_P(0) \leq \frac{1}{8} |A|^2 .\]
Because \(1_L = 1 - 1_P\) we have
\[1_{a_1\cdot A_1}* 1_{a_2\cdot A_2} * 1_L(0) \geq \Bigl(1-\frac{1}{8}\Bigr)|A|^2\geq \Bigl(1-\frac{1}{8}\Bigr)|A_1||A_2|.\]
Thus we can apply Lemma~\ref{lemma:almost_periods_to_increment} with \(\epsilon = \frac{1}{8}\) to finish the first case of the proof.\\
For the second case let us assume that \(|P|\leq | B'_{1+(k-3)\delta}|/2\).\\ We will proceed in a similar fashion, however this time we will apply Corollary~\ref{cor:croot_sisask} to the sets \(a_4 \cdot A_4, \cdots, a_k\cdot A_k\), \(M:=a_3 \cdot A_3\) and $L:=-P$. The Bohr set will be \(B'_{1+(k-3\delta)}\) as previously.
We use almost the same $S$ as in the previous case. We only swap \(a_2\cdot A_2\) for \(a_4\cdot A_4\), that is \(S:= a_4\cdot B'_{\delta\nu}\). Arguing in exactly the same way we have 
\[|a_4\cdot A_4 + S| \leq \frac{16}{7\alpha}|a_4\cdot A_4|.\]
This time we have 
\(|L| \leq |B'_{1 + (k-3)\delta}|/2 \leq \frac{2}{3 \alpha}|a_3\cdot A_3|,\)
again arguing in the same way, just swapping \(a_1\cdot A_1\) for \(a_3\cdot A_3\).\\
We also estimate the \(||\cdot|| _1\) norm, notice that
\begin{equation*}
\begin{split}
||1_{a_4\cdot A_4}*\cdot 1_{a_5\cdot A_5} *\cdots * 1_{a_k \cdot A_k } *1_{-P}||_1 &= |A_4||A_5|\cdots |A_k||L|\\
&\leq |A_3||A_4||A_5|\cdots |A_k| /(2\alpha).
\end{split}
\end{equation*}
So this time our function in the application of Lemma 1 will be
\[1_{a_4\cdot A_4}*\cdot 1_{a_5\cdot A_5} *\cdots * 1_{a_k \cdot A_k } *1_{-P}/ |A_4||A_5|\cdots |A_k|.\]
It remains to estimate (using $f$ defined above) \(f*1_{-P}(0)\). We see that
\[f*1_{-P}(0)=\sum_{p\in P}f(p) = |A_3||A_4|\cdots|A_k| - \sum_{p\notin P}f(p),\]
however, by the definition of $P$, we have
\[\sum_{p\notin P}f(p) \leq \frac{\alpha}{8}|B'_{1+(k-3)\delta}||A_4||A_5|\cdots |A_k|\leq \frac{1}{7}\cdot 1.01 \cdot|A_3||A_4|\cdots |A_4|,\]
where we use the fact that density of \(a_3\cdot A_3\) in \(B'\) is at least \(\frac{7}{8}\alpha\).
Combining the two previous lines we obtain
\[f*1_{-P}(0)\geq \Bigl(1-\frac{1}{6}\Bigr)|A_3||A_4|\cdots|A_k|,\]
which finally lets us apply Lemma~\ref{lemma:almost_periods_to_increment} with \(\epsilon = \frac{1}{6}\).\\
We actually considered 3 possible cases, one being the second conclusion of Lemma~\ref{lemma:coefficients}. That one had by far the worst density increment, which we record in the current lemma. This finishes the proof.
\end{proof}
We can finally prove Theorem~\ref{thm:number_of_solutions}.
\begin{proof}
Let $A^{(0)} = A$ and $B^{(0)} = \mathbb{Z}/p\mathbb{Z}$. As mentioned before, we pick
\[p > (|a_1|+|a_2|+\cdots+|a_n|) N.\]
We iterate Lemma~\ref{lemma:increment_on_bohr_set} on these sets, obtaining $(A^{(1)}, B^{(1)})$, \((A^{(2)}, B^{(2)}), \cdots\). We know that after, say, $s$ steps it is no longer possible. That is because the density of \(A^{(i)}\) cannot exceed $1$. Since Lemma~\ref{lemma:increment_on_bohr_set} cannot be applied to $A^{(s)}$ and $B^{(s)}$ we know that $A^{(s)}$ contains at least \(e^{-Cd_s\log(d/\alpha)}|A^{(s)}|^{k-1}\) solutions of \(a_1 x_1 + a_2 x_2 + \cdots +a_k x_k =0\).
We easily calculate that
\begin{align*}
s &\leq C \log(1/\alpha),\\
d_s &\leq C\log^5(2/\alpha),\\
\rho_s &\geq (c\alpha)^{Cs},\\
\end{align*}
and so
\[|A^{(s)}|\geq \alpha |B^{(s)}| \geq \alpha (\rho_s/2\pi)^{d_s}N \geq e^{-C\log^7(2/\alpha)}N.\]
We also note that \(\log d_s \ll \log \log ^5(2/\alpha)\ll \log(2/\alpha)\). Putting all of these bounds together with the estimate on the number of solutions in $A^{(s)}$ we have
\[e^{-Cd_s\log(d_s/\alpha)}|A^{(s)}|^{k-1} \geq e^{-C\log^7(2/\alpha)}N^{k-1} .\]
So $A$ contains at least \(e^{-C\log^7(2/\alpha)}N^{k-1}\) solutions to \(a_1 x_1 + a_2 x_2 + \cdots +a_k x_k =0\), since \(A^{(s)}\subseteq A\).
\end{proof}
\section{Behrend-type construction for the lower-bound}
In this section we prove Theorem~\ref{thm:behrend_construction}, which gives an analogous lower bound to what we have proved in the last section. We modify the argument of Tao (\cite{behrend_tao}, Proposition 1.3) to show a Behrend-type bound for convex equations.
\begin{proof}
    Let \(N =  M^{d+d'}\), where \(d'\geq 0\) is an arbitrary integer and $M$, $d$ will be chosen later. Define a map \(D:[N]\rightarrow [M]^d\) to be the mapping that sends a number to the vector of its last $d$ digits in base $M$. To be precise, we define $D$ as 
    \[D(n)_i = \Bigl\lfloor\frac{n}{M^{i-1}}\Bigr\rfloor \mod M \text{ for } 1\leq i\leq d.\]
    Define \(T\subseteq [N]\) by including all $n$ such that for all $i$ there is \(D(n)_i < \frac{M}{k}\). Clearly \(|T|\gg N\cdot k ^{-d}\). Among the numbers \(1, 2\cdots, dM^2\) choose $r$ such that the sphere \(||D(x)||_2^2=r\), which we call $A$, has at least \(\frac{|T|}{dM^2} \gg \frac{1}{d k^d}M^{d'+d-2}\) points inside \(T\). Suppose that \(x_1,x_2,\cdots,x_k\in A\) are such that \(x_1+\cdots+x_{k-1}=(k-1)y\). Since there is no carry-over in base $M$ when adding elements of $A$ up to $k$ times, we have
    \[D(x_1)+\cdots+D(x_{k-1})=(k-1)D(x_k).\]
    This however, can only be the case when \(D(x_1)=\cdots=D(x_{k-1})=D(x_k)\) by convexity, since all of the points belong to a sphere of radius $r$.
    We conclude that there is at most \(|A|M^{d'(k-2)}\) solutions to the equation \(x_1+\cdots+x_{k-1}=(k-1)x_k\) inside $A$.
    If for a small constant \(c>0\) we set \(d:=c\log(2/\alpha)\) and \(M:=\alpha^{-c}\) we have
    \begin{equation*}
    \begin{split}
    |A|/N &\gg \frac{M^{d'+d-2}}{Nd k^d } = \frac{1}{d k^d M^2}\\
    &\geq \frac{1}{c\log(2/\alpha)k^{c\log(2/\alpha)}\alpha^{-2c}}\geq \alpha.
    \end{split}
    \end{equation*}
    Moreover, bounding the size of $A$ by \(N\) we have
    \[\frac{|A|M^{d'(k-2)}}{N^{k-1}}\leq \frac{|A|}{M^{(k-1)d} + d'}\leq M^{d'+d-(k-1)d-d'} = M^{-(k-2)d} \ll e^{-c\log^2(2/\alpha)},\]
    which is the desired maximal number of solutions.
\end{proof}
\section{Improving the bound for many variables}
Theorem~\ref{thm:number_of_solutions} and the analogous result of Schoen and Sisask~\cite{schoen} give the relevant constant $7$ in the bound (for example \(e^{-C\log^7(2/\alpha)}N^{k-1}\) in Theorem~\ref{thm:number_of_solutions}). Behrend-type construction in Theorem~\ref{thm:behrend_construction} shows that this cannot be improved to more than~$2$. In this section we show how to bring the constant down almost to $6 $, provided the considered equation is long enough. By the end of this section this is summed up as the proof of Theorem~\ref{thm:strong_no_solutions}. The main idea is Theorem~\ref{thm:bohr_set_in_sumset} below, which allows us to find a large Bohr set within \(wA-wA\) for some $w$. After that a density increment can be obtained quite easily. The approach builds on an idea by Konyagin. To our knowledge it was not published, but is mentioned by Sanders in his survey~\cite{sanders_revisited}.
\begin{theorem}\label{thm:bohr_set_in_sumset}
Let \(A\subseteq B\) with \(|A|=\alpha |B|\) where $B$ is a regular Bohr set of dimension $d$ and width \(\rho\). Let \(m\geq 1\). There exists a Bohr set \(\Tilde{B}\subseteq 3^{m+1} A-3^{m+1} A\) of dimension $d+d'$ and width \(\rho_{\Tilde{B}}\). Moreover, $\Tilde{B}$ can be chosen so that
\[d' \leq C\log^{3+\gamma}(2/\alpha),\]
where \(\gamma=2^{1-m}\)
and 
\[\rho_{\Tilde{B}} \geq \rho\frac{c\alpha^3}{d^5d'}.\]
\end{theorem}
\begin{proof}
    The plan is to  apply Theorem~\ref{thm:croot_sisask_modified} on inductively constructed sets \(A'_m\) and \(T_m\). The resulting Bohr set will have significantly smaller dimension than the one we get by naively applying Theorem~\ref{thm:croot_sisask_modified} to the set $A$.
    
Define constants \(k_0,k_1, k_2,\cdots, k_m\) to be
\[k_{i} := \bigl\lceil\log^{1-2^{-i}}(2/\alpha)\bigr\rceil,\]
for \(1\leq i\leq m\). Let \(k=k_0+k_1+\cdots k_m\), then \(k\ll m\log(2/\alpha)\). 
    We choose \(\delta \geq \frac{\alpha}{Cd}\) to get regular Bohr set \(B_{\delta}\) such that \(|B_{1+2\delta}|\leq \frac{3}{2} |B|\). By Lemma~\ref{lemma:high_density_small_bohr_set} there exists $x$ such that \(A'_0 := A\cap (x+B_{\delta})\) has density at least \(0.9\alpha\) within \(x+B_{\delta}\).
    Similarly, we choose \(\frac{\alpha^2}{Ckd^2}\leq \nu\leq \delta/k\) to get regular Bohr set \(B_{\delta\nu}\) such that \(|B_{\delta (1 + 2 k \nu)}|\leq \frac{3}{2}|B_{\delta}|\). Again, by Lemma~\ref{lemma:high_density_small_bohr_set} there exists $x'$ such that \(T_0:=A\cap (x'+B_{\delta\nu})\) has density at least \(0.8\alpha\) within \(x' + B_{\delta\nu}\). We see that
    \[|A+A_0'|\leq |B + (x + B_{\delta})|\leq |B_{1+\delta}|\leq \frac{3}{2}|B|\leq \frac{2}{\alpha}|A|\]
    and so
    \[\eta := |A+A'_0|/|A| \leq \frac{2}{\alpha}.\]
    Similarly we have
    \[|A'_0+ T_0|\leq |B_{\delta(1+\nu)}| \leq \frac{3}{2} |B_{\delta}|\leq \frac{3}{2 \cdot 0.9 \cdot \alpha} |A'_0| \leq \frac{2}{\alpha} |A_0'|.\]
    We will show how to inductively construct sets \(A'_i\) and \(T_i\), for which the following conditions hold. The properties deduced above serve as the base case.
    \[T_i\subseteq T_{i-1}\tag{4}\]
    \[A_{i-1}'\subseteq A_i'\subseteq B_{\delta(1+k\nu)}\tag{5}\]
    \[|A+A_i'|\leq \frac{2}{\alpha}|A|\tag{6}\]
    \[|A_i'+T_{i-1}|\leq \Bigl(\frac{2}{\alpha}\Bigr)^{1/k_{i-1}}\tag{7}\]
    \[|T_i| \geq \exp(-Ck_{i}^2/k_{i-1}\log^2(2/\alpha))|T_{i-1}| \tag{8}\]
    \[k_i T_i\subseteq A + A'_{i-1} - A - A'_{i-1}\tag{9}\]
    We apply Corollary~\ref{cor:croot_sisask} to the sets \(A_{i-1}', A, -(A+A_{i-1}')\) and \(T_{i-1}\) with \(\epsilon = 1/2\) and the chosen $k_i$ to obtain a set of periods \(T_i\subseteq T_{i-1}\) where condition (8) holds
    such that for every \(t\in k_iT_i-k_iT_i\) there is
\[|1_A*1_{A'_{i-1}}*1_{-(A+A'_{i-1})}(t)- 1_A*1_{A'_{i-1}}*1_{-(A+A'_{i-1})}(0)|\leq \frac{1}{2} |A||A'_{i-1}|.\]
Notice that \(1_A*1_{A'_{i-1}}*1_{-(A+A'_{i-1})}(0)=|A||A'_{i-1}|\), thus by the triangle inequality
\[1_A*1_{A'_{i-1}}*1_{-(A+A'_{i-1})}(t) \geq \frac{1}{2}|A||A'_{i-1}| >0\]
and so (9) holds. We also have 
\[|A'_{i-1}+k_iT_i|\leq |B_{\delta(1+k\nu)} + B_{\delta(k_i\nu)}| \leq |B_{\delta(1+2k\nu)}|\leq \frac{3}{2}B_{\delta}\leq \frac{2}{\alpha}|A'_0|\leq \frac{2}{\alpha}|A'_{i-1}|.\]
Since by adding $T_i$ $k_i$ times we increase $A'_{i-1}$ by the factor of \(\frac{2}{\alpha}\), there must be an \(0\leq l_i < k_i\) such that
\[|A_{i-1}'+l_iT_i+T_i|\leq \Bigl(\frac{2}{\alpha}\Bigr)^{1/k_i}|A'_{i-1}|\leq \Bigl(\frac{2}{\alpha}\Bigr)^{1/k_i}|A'_{i-1} + l_iT_i|.\]
Define \(A'_i:=A'_{i-1}+l_{i}T_{i}\) so that (7) is satisfied. Moreover, the first part of (5) holds trivially and the second part is true because
\[A_i = A'_0 + l_1T_1+\cdots+l_i T_{i}\subseteq A'_0 + kT_0 \subseteq B_{\delta(1+k\nu)}.\]
Let us also notice that
\[|A+A'_i|\leq |B_{1+\delta(1+k\nu)}|\leq|B_{1+2\delta}|\leq \frac{3}{
2} |B|\leq \frac{2}{\alpha} |A|.\]
Therefore (6) holds and the inductive step is complete.

We now calculate the closed form of the recursive relation (9), making use of the fact that we defined \(A_i'\) so that
\[A_i' = A_{i-1}' + l_{i} T_{i} \subseteq A_{i-1}' + k_i T_i \subseteq 2 A_{i-1}' + A - A - A_{i-1}'.\]
Let \(n_i = (3^i-1)/2\), then by simple induction we have
\[A_i'\subseteq (1+n_i) A_0' - n_i A_0' + n_i (A-A).\]
Since \(A_0'\subseteq A\) we can write it as
\begin{equation*}
\begin{split}
A_i'  &\subseteq (1+3n_i) A - (1+3n_i) A\\
&= n_{i+1} A - n_{i+1}A.    
\end{split}
\end{equation*}

Iterate the above inductive procedure $m$ times to obtain the sets \(T_1, T_2,\cdots, T_m\).
Recall that \(\gamma=2^{1-m}\). For every \(i\geq 1\) we have
\begin{equation*}
\begin{split}
\frac{k_{i}^2}{k_{i-1}}&= \frac{\bigl\lceil\log^{1-2^{-i}\gamma}(2/\alpha)\bigr\rceil^2}{\bigl\lceil\log^{1-2\cdot2^{-i}\gamma}(2/\alpha)\bigr\rceil} \leq 
\frac{\bigl(\log^{1-2^{-i}\gamma}(2/\alpha)+1\bigr)^2}{\log^{1-2\cdot2^{-i}\gamma}(2/\alpha)}\\
&= \log(2/\alpha) + 2\log^{2^{-i}\gamma}(2/\alpha) + \log^{2\cdot2^{-i}\gamma-1}(2/\alpha)\\ &\leq \log(2/\alpha) + 3 \leq 4\log(2/\alpha).
\end{split}
\end{equation*}

Thus we can give the lower bound
\begin{equation*}
\begin{split}
|T_m|&\geq  \exp\bigl((-k_1^2-k_2^2/k_1-k_3^2/k_2-\cdots - k_m^2/k_{m-1})(C\log^2(2/\alpha))\bigr)|T_0|\\
&\geq\exp\bigl(-Cm \log^3(2/\alpha)\bigr)|T_0|.
\end{split}
\end{equation*}
Let us apply Theorem~\ref{thm:croot_sisask_modified} to sets \(A, A'_m, -(A+A'_m), T_m\) with \(\epsilon = 1/2\), making use of the properties (6) and (7). This way we find a Bohr set \(\Tilde{B}\) such that 
\begin{equation*}
\begin{split}
\Tilde{B} &\subseteq A+A'_m-(A+A'_m)\subseteq (2n_{m+1}+1)(A-A) = 3^{m+1}A-3^{m+1}A,\\
\dim \Tilde{B} &= d+d' \leq d + C\log^4(2/\alpha)/k_{m-1} + C\log(1/\sigma),\\
\sigma &= |T_m|/|B_{\delta\nu}|\\ &\geq \exp(-C m\log^3(2/\alpha)) |T_0|/|B_{\delta\nu}|\\ &\geq \exp(-C m\log^3(2/\alpha))\cdot 0.8\alpha .
\end{split}
\end{equation*}
We notice that since \(k_{m-1}\gg \log^{1-2^{1-m}}(2/\alpha)\), by setting \(\gamma = 2^{1-m}\) we have
\[\dim \Tilde{B}\ll \log^{3+\gamma}(2/\alpha) + \log^3(2/\alpha),\]
moreover
\[\rho_{\Tilde{B}} = \rho\delta\nu \frac{(2/\alpha)^{1/2}}{2d^2 d'}\geq\rho\frac{c\alpha^3}{d^5d'},\]
which are the desired bounds.
\end{proof}
Notice that Theorem \ref{thm:bohr_set_in_sumset} works also when $A$ is contained in a translate of a Bohr set \(g+B\), for some \(g\in G\). To see that it is enough to consider \(A-g \subseteq B\).\\
We now show how to use the Bohr set from Theorem~\ref{thm:bohr_set_in_sumset} to obtain a density increment for solution free sets. The strategy is very similar to the one suggested by Schoen and Shkredov~\cite{many}: we observe that certain translates of Bohr sets cannot intersect $A$, as this would lead to a non-trivial solution. By an averaging argument, $A$ must have higher density in the remaining translates.
\begin{lemma}\label{lemma:strong_density_increment}
Let \(m\geq 1\) and \(k\geq 2\cdot3^{m+1}+2\). Let \(A\subseteq B\), where \(|A|=\alpha |B|\) and $B$ is a Bohr set of dimension $d$ and width \(\rho\). Assume that \(|B|\geq \alpha^{-2}\Bigl(\frac{Cd^2}{\alpha}\Bigr)^{3d}\). Suppose that $A$ does not contain any non-trivial solutions to the equation
\[x_1+x_2+\cdots + x_{k-1} = (k-1) x_k.\]
Then, there is a Bohr set \(\Tilde{B}\) of dimension \(d+d'\)and radius \(\rho_{\Tilde{B}}\), such that for some $y$ we have \(|\Tilde{B}\cap (A-y)|\geq 1.01 \alpha|\Tilde{B}|\). It is possible to choose it in such a way that 
\[d'\leq C\log^{3+2^{1-m}}(2/\alpha)\]
and
\[\rho_{\Tilde{B}} \geq \rho\frac{c\alpha^3}{d^5d'}.\]
\end{lemma}
\begin{proof}
Applying Lemma~\ref{lemma:coefficients} we find a translate of a Bohr \(t+B'\) set of dimension $d$ and radius \(\delta\geq \frac{c\alpha}{d}\) and sets \(A_1,A_2,A_3\subseteq A\),
such that \((k-1)A_1\), \(-A_2\), \(A_3\) have densities at least \(\frac{7}{8}\alpha\) inside \(t+B'\), \(t+B'\), \(t+B'_{\delta_1}\), or there is a density increment \(1+1/48\) on one of these sets. Without loss of generality we assume that \(t=0\), since our equation is translation-invariant. By choosing suitable \(\delta_1, \delta_2\geq \frac{1}{Cd}\) we ensure that \(|B'_{1+2k\delta_1}|\leq 1.01 |B'|\) and \(|B'_{\delta_1+\delta_1\delta_2}|\leq 1.01|B'_{\delta_1}|\).\\
By an averaging argument, we find translate \(t_2+B'_{\delta_1\delta2}\) for \(t_2\in B'_{\delta_1}\), such that \(A'_3 = A_3\cap(t_2+B'_{\delta_1\delta_2})\) has density at least \(0.8\alpha\) inside \(t_2+B'_{\delta_1\delta2}\).
Consider the sum 
\[\sum_{x\in B'_{\delta_1+\delta_1\delta_2}} 1_{A_3}*1_{A_3'}(x) = \sum_{x\in A_3+A_3'}|A_3\cap (x-A_3')| = |A_3||A_3'|.\]
Thus we know that for some \(x\in B'_{\delta_1+\delta_1\delta2}\) we have \[|A_3\cap (x-A_3')|\geq |A_3||A_4|/|B'_{\delta_1+\delta_1\delta2}|\geq \frac{1}{2}\alpha^2|B'_{\delta_1}|.\] Define \(A_3^*:=A_3\cap (x-A_3')\), clearly \(A_3^*\subseteq x+B'_{\delta_1\delta_2}.\)
We construct \(A_{3}^+\) by inserting all elements of \(A_3^*\) to \(A_{3}^+\) unless an element \(a\in A_3^*\) is already in \(x-A_{3}^+\), then we add it to \(A_{3}^-\). Clearly the sizes of \(A_{3}^+\) and \(A_{3}^-\) differ by at most 1 and \(x-A_{3}^-\subseteq A_{3}^+\).\\
To show that \(A_{3}^+\) and \(A_{3}^-\) are non-empty we need to know that the size of \(A_3^*\) is at least 2. Using our assumption on the size of \(|B|\) we have
\[|A_3^*|\geq \frac{1}{2}\alpha^2|B_{\delta\delta_1}|\geq \frac{1}{2}\alpha^2\Bigl(\frac{\alpha}{Cd^2}\Bigr)^{3d}|B|\geq 2.\]
\\ Set \(w=\Bigl\lfloor\frac{k-2}{2}\Bigr\rfloor\). This way we have \(w\geq 3^{m+1}\). We initially assume that $k$ is even and the floor function is unnecessary. 

    Clearly \(A_{3}^+\subseteq B'_{\delta_1}\) and it has density at most \(\frac{1}{4} \alpha^2\), moreover \(wA_{3}^+-wA_{3}^+ \subseteq w A_{3}^++wA_{3}^--wx\).
    By Theorem~\ref{thm:bohr_set_in_sumset} we find a Bohr set \(T\subseteq wA_{3}^+-wA_{3}^+\) of the desired width and dimension.
    Suppose that \(a\in (k-1)A_1\) and \(b\in A_2\). Then we must have \((2w+1)b-a\notin T+wx\) as otherwise we would have \((2w+1)b-a\in w A_{3}^++wA_{3}^-\) and that would mean a non-trivial solution to the equation \[x_1+x_2\cdots + x_{2w+1}=(2w+1)x_{2w+2}.\]
    Consider a larger Bohr set \(B^* = B'_{1+w(\delta_1+\delta_1\delta_2)}\), we have \(|B^*|\leq |B'_{1+2k\delta_1}|\leq 1.01|B'|\). Thus \((k-1)A_1, A_2\) have densities at least \(0.8\alpha\) inside  \(B'\subseteq B^*+wx\).
    
    At this point lets remark what happens if \(k\) is odd. Let $z$ be an arbitrary element of \(A\cap B'_{\delta_1}\) Then instead of \((2w+1)b-a\notin T+wx\) we assert that \((2w+2)b-a\notin T+wx+z\), thus adding one extra variable to our equation, which we set immediately to $z$. If we choose \(B^{**}=B'_{1+(w+1)(\delta_1+\delta_1\delta_2)}\), we still have \(B'\subseteq B^{**}+(w+1)x\) and the rest of the argument remains the same.
    
    By the above observation about non-inclusion we notice that if \(y\in x+B^*\) then either \((y+T_{1/2})\cap ((k-1)A_1)\) or \((y-T_{1/2})\cap (-A_2)\) must be empty. Summing over all such $y$ we have
    \[1.6\alpha|B^*||T_{1/2}|\leq \sum_{y\in x+B^*}|(y+T_{1/2})\cap ((k-1)A_1)| + |(y-T_{1/2})\cap (-A_2)|.\]
    Because one element in the sum is always equal to $0$ we must have some \(y\in x+B^*+wx\) for which 
    \[\frac{1.6\alpha|B^*||T_{1/2}|}{1.01|B^*|}\leq |(y+T_{1/2})\cap ((k-1)A_1)|\]
    or
    \[\frac{1.6\alpha|B^*||T_{1/2}|}{1.01|B^*|}\leq |(y-T_{1/2})\cap (-A_2)|.\]
    That is almost a density increment on a translate of $T_{1/2}$. After multiplying the set of characters of $T_{1/2}$ either by $k-1$ or $-1$ we obtain a density increment of $1.5$ on the resulting Bohr set \(\Tilde{B}\).
\end{proof}
We are now in position to prove Theorem~\ref{thm:strong_no_solutions}.
\begin{proof}
We proceed in a similar way as in the proof of Theorem~\ref{thm:number_of_solutions}. Again we take \(p > (|a_1|+|a_2|+\cdots+|a_n|) N\).
Let $A^{(0)} = A$ and $B^{(0)} = G$. We iterate Lemma~\ref{lemma:strong_density_increment} on these sets, obtaining $(A^{(1)}, B^{(1)})$, \((A^{(2)}, B^{(2)}), \cdots\). We know that after, say, $s$ steps it is no longer possible. That is because the density of \(A^{(i)}\) cannot exceed $1$.
Clearly 
\[s\leq C \log(2/\alpha)\]
and we easily calculate that
\[d_s\leq C\log^{4+\gamma_m}(2/\alpha),\]
\[\rho_s \geq (c\alpha)^{Cs}.\]
The only reason for density increment not possible is that \(|B^{(s)}|< \alpha^{-2}\Bigr(\frac{Cd_s^2}{\alpha}\Bigr)^{3d_s}\). On the other hand we can lower-bound the size of \(B^{(s)}\) by Proposition 3. Comparing the lower-bound and the upper-bound we have
\[(\rho_s/2)^{3d_s} N \leq \alpha^{-2}\Bigr(\frac{Cd_s^2}{\alpha}\Bigr)^{3d_s},\]
which up to a constant the same as
\[\log(N) \leq 3d_s\log\Bigl(\frac{Cd_s^2}{\alpha \rho_s}\Bigr).\]
Substituting the bounds for \(d_s\) and \(\rho_s\) that is equivalent (again up to a constant) to
\[\log(N) \leq C\log^{6+\gamma_m}(2/\alpha).\]
Rearranging, we obtain
\[\alpha \leq e^{-c\log^{1/(6+\gamma_m)} N}.\]
\section{applications}
Theorem \ref{thm:number_of_solutions_bloom} of Bloom has been used in a number of papers employing Fourier Transference Principle. Examples of such results are papers from Prendiville~\cite{prend}, Chow~\cite{chow}, Browning and Prendiville~\cite{browning}. We think that substituting Theorem \ref{thm:number_of_solutions_bloom} by our Theorem \ref{thm:number_of_solutions} for equations of length 4 and more, could bring quantitative improvements. We briefly recall the first of the results~\cite{prend} and state how the bound improves.

Let \(S\in \{1,2,\cdots, N\}\) and suppose that the only solutions to the equation
\[x_1 + y_1 = x_2 + y_2\]
for \(x_1,x_2,y_1,y_2\in S\) are trivial. Then $S$ is called a Sidon set and it is known that \(|S|\leq (1+o(1))N^{1/2}\). The problem of finding solutions to invariant equations in Sidon sets was considered by Conlon, Fox, Sudakov and Zhao~\cite{cfsz}. They give a weak upper bound of \(|S|\leq o(N^{1/2})\), providing a comment about how to use their methods to obtain a stronger bound. Prendiville~\cite{prend} used the method of Fourier Transference Principle to get an improvement on the work of Conlon, Fox, Sudakov and Zhao.
\begin{theorem}(Prendiville)
If \(N\geq 3\) and \(S\subseteq \{1,2,\cdots N\}\) is a Sidon set lacking solutions to an invariant equation \(a_1 x_1+a_2 x_2 + \cdots + a_k x_k=0\) in \(k\geq 5\) variables, we have
\[|S|\leq C N^{1/2}(\log\log N)^{-1}.\]
\end{theorem}
By inspecting the proof of Prendiville~\cite{prend} we see that we can improve the bound by substituting Theorem \ref{thm:number_of_solutions_bloom}, which is used there, with Theorem \ref{thm:number_of_solutions}. As the result, we can prove the following theorem.
\begin{theorem} \label{thm:sidon_improved}
Let \(S\subseteq \{1,2,\cdots, N\}\) be a Sidon set, which contains no non-trivial solutions to an invariant equation \(a_1 x_1+a_2 x_2 + \cdots + a_k x_k=0\) in \(k\geq 5\) variables. Then we have
\[|S| \leq N^{1/2}\exp(-C(\log\log N)^{1/7}).\]
\end{theorem} 

\end{proof}

\bigskip

\textsc{Faculty of Mathematics and Computer Science, Adam Mickiewicz University, Umultowska 87, 61-614 Poznan, Poland }

\textit{Email address:} tomasz.kosciuszko@amu.edu.pl
\end{document}